\documentclass[letterpaper, 10 pt, conference]{ieeeconf}  

\usepackage[T1]{fontenc}

\IEEEoverridecommandlockouts                              

\usepackage{graphics} 
\usepackage{epsfig} 
\usepackage{times} 
\usepackage{amsmath}
\usepackage{amssymb}  
\usepackage{tabu}
\usepackage{color}
\usepackage{mathptmx} 
\usepackage{comment}
\usepackage[compress]{cite}
\DeclareMathOperator{\E}{\mathbb{E}}

\usepackage{enumerate}
\usepackage{bbm}
\usepackage{hyperref}
\newtheorem{cor}{Corollary}
\newtheorem{remark}{Remark}
\newtheorem{pro}{Proposition}

\newtheorem{thm}{Theorem}
\newtheorem{asu}{Assumption}

\ifodd 1

\newcommand{\com}[1]{{\color{red}{Comment: #1}}}
\else

\newcommand{\com}[1]{}
\fi

\title{\LARGE \bf
A Reputation-Based Contract for Repeated Crowdsensing with\\ Costly Verification
}

\author{Donya G. Dobakhshari, Parinaz Naghizadeh, Mingyan Liu, and Vijay Gupta
\thanks{D. G. Dobakhshari and V. Gupta are with the Department of Electrical Engineering, University of Notre Dame, IN, USA. Email:{\tt\small (dghavide, vgupta2)@nd.edu}. They are supported in part by NSF grants 1239224, 1544724, and 1550016.}
\thanks{P. Naghizadeh and M. liu are with the Department of Electrical Engineering and Computer Science, University of Michigan, MI, USA.  Email:{\tt\small (naghizad, mingyan)@umich.edu}. They are supported in part by NSF grants ECCS-1446521 and CNS-1646019.}}%

\begin{document}
\maketitle
\thispagestyle{empty}
\pagestyle{empty}

\begin{abstract}
We study a setup in which a system operator hires a sensor to exert costly effort to collect accurate measurements of a value of interest over time. At each time, the sensor is asked to report his observation to the operator, and is compensated based on the accuracy of this observation. Since both the  effort and observation are  private information for the sensor, a naive payment scheme which compensates the sensor based only on his self-reported values will lead to both shirking and falsification of outcomes by the sensor. 
We consider the problem of designing an appropriate compensation scheme to incentivize the sensor to at once exert costly effort and truthfully reveal the resulting observation.  

To this end, we formulate the problem as a repeated game and propose a compensation scheme that employs stochastic verification by the operator coupled with a system of assigning  reputation to the sensor. 
In particular, our proposed payment scheme compensates the sensor based on both the effort in the current period  as well as the history of past behavior. 
We show that by using past behavior in determining present payments, the operator can both incentivize higher effort as well as more frequent truthtelling by the sensor and decrease the required verification frequency.
\end{abstract}

\section{Introduction}

Providing incentives for individuals to follow the system operator's desired policies in smart networks has received increased attention in recent years (see, e.g., ~\cite{gao, saad, cassand, manshaei} and the references therein). In a typical such setting, the system operator delegates tasks to several autonomous agents. 
The agents may not benefit directly from the outcome of the task, and hence may not exert sufficient effort to complete the task with the quality desired by the operator. Further, due to reasons such as privacy, it may be costly (or even impossible) for the system operator to access information such as effort exerted by each agent or the resulting outcome directly; instead, she must rely on the data reported by the agent. Consequently, in the absence of a suitable compensation scheme, rational self-interested agents may refrain from exerting desired effort, and further possibly send falsified information in response to the operator's inquiry about the task.

As our running example, we consider a crowdsensing application in which autonomous  sensors are employed to take measurements about a quantity of interest to the system operator over a predetermined time horizon. At each time step, these measurements are used by the system operator to generate an estimate of the quantity of interest. The sensors incur an effort cost for obtaining each measurement with a specified level of accuracy. A sensor can choose to increase his effort and attain more accurate measurements, but at the expense of a higher effort cost. This cost may model, e.g.,  the cost of operating the device, or battery power. The sensors are then rewarded based on the accuracy of the information they provide to the operator. Since sensors do not, in general, attach a value to the accuracy of the estimate at the system operator, they do not have any incentive to exert costly effort for generating accurate observations. Furthermore, the system operator  has no access to the sensors' private information (i.e., either the level of effort, or the true accuracy of the measurements they generate). As a result, with a compensation scheme that rewards sensors for self-reported accuracy of the measurements they generate, the sensors may expend little effort, yet misreport their accuracy in order to receive higher compensation. The problem we consider in this paper is to generate a contract for the sensors so that they provide measurements with sufficient accuracy to enable the operator over time to generate an estimate with a desired quality. While we concentrate on crowdsensing as the example of interest in this paper, the contract can obviously be applied to other examples mentioned earlier that face the same challenges. 


Intuitively, designing an appropriate contract in our setting is difficult due to two reasons: (i) profit misalignment, and (ii) information asymmetry between the system operator and the sensor providing the information. To alleviate these challenges, the operator needs to design incentive mechanisms that mitigate both \emph{moral hazard} (i.e., incentivizing desired actions when effort is costly and the level of effort expended is private information for a participant, e.g.,~\cite[Chapter~4]{laf}) and \emph{adverse selection}  (i.e., incentivizing participants to provide truthful information when information is private to them, e.g.,~\cite[Chapter~3]{laf}). 
 While an extensive literature in contract theory (see, e.g.,~\cite{gao, laf, salanie} and the references therein for an overview of the subject) has focused on resolving either moral hazard or adverse selection separately, we consider the problem of moral hazard \textit{followed} by adverse selection in a \textit{repeated setting}. This problem has received much less attention in the literature.  Notable exceptions include~\cite{dasg,bohren20,crocker}, and~\cite{ ghavidel}, which discuss the problem of moral hazard followed by the adverse selection in a \textit{static} framework.  The main difference of our work is to consider this problem in a repeated setting. The repeated setting provides new challenges to the problem in that sensors may adopt time-varying strategies to gain, and then misuse, the trust of the system operator. Further, the solutions in~\cite{dasg,bohren20,crocker,ghavidel} rely on verification of the outcomes generated by the agents {(either direct verification or cross-verification)}. If such verification is costly, care must be taken in the repeated setting to bound the verification cost by allowing only infrequent verification. 

 We would like to explain the concept of verification here. In order to mitigate the information asymmetries of moral hazard followed by adverse selection, the operator must find some way to verify the sensor's self-reported outcomes. Verification is crucial, as the operator relies only on the information transmitted by the sensor (i.e., she has no additional measurements or observations of the value of interest) and compensates the sensor for these measurements over each time step. 
  Therefore,  verification of the sensor's self-reported values is crucial in settings with information asymmetry (see e.g., ~\cite{mookherjee, border, baron}, in which verification is similarly proposed). However, the problem is not trivial even with verification, since verification may be delayed, costly, or imperfect. In this paper, we assume that verification is costly. Thus, while the operator can verify the sensor at each stage and through appropriate penalties ensure  truthful revelation by the sensors, this approach may not, in general, be optimal for the operator in the sense of optimizing his total utility due to costly nature of verification. Intelligent use of verification is, in itself, not a trivial matter.

We propose an alternative scheme, in which we relax the requirement for verification at every time step. We instead assume that verification is done with a specified  probability at every time step. We use this verification scheme  coupled with a \emph{history-dependent} payment scheme to then design an optimal contract. In particular, we base the compensation on a \emph{reputation} score assigned to each sensor. This reputation is based on the history of the  past interactions of the sensor with the operator. In other words, the operator rewards the sensor based on the sequence of his actions, rather than merely his behavior at the current stage.  In particular, the operator assigns
higher reputation (and consequently higher payments) to a
sensor that is verified and detected to be honest.  The operator uses reputation as well as verification  to incentivize the sensor to exert sufficient effort, and also to not misreport his outcome, so as to optimize his own utility over multiple time steps.  


It is worth mentioning that using reputation for mitigating information asymmetry,  particularly in repeated games, is a popular strategy in the literature, see e.g.~\cite{naghizad, mailath2, abreu,bohren, della}. 
In~\cite{ naghizad}, the authors address the problem of mitigating pure adverse selection by means of reputation indices in a static setting. 
 The authors of \cite{mailath2} present a comprehensive study of the use of reputation for mitigating adverse selection in repeated games, i.e., learning the hidden types of players  through repeated interactions. 
The works in ~\cite{abreu,bohren, della, han2016} 
have focused on repeated interactions between a system operator and an agent under the assumption of only hidden action (moral hazard) for the agent.  The work in \cite{abreu} proposes a method for mitigating moral hazard using imperfect verification, while \cite{bohren} additionally introduces reputations.

{ In the crowdsensing literature, the work studied in \cite{zhang2012, radanovic, shnay, rada} are also related to the work presented in this paper. \cite{zhang2012}  considers an ex ante payment for the agents with verifiable outcome and mitigates the pure moral hazard problem by introducing reputation-based incentive. Further, a class of peer prediction methods  has been studied in \cite{radanovic, shnay, rada}. Peer prediction mechanisms mostly consider the perfect knowledge of actions by the agents and deal with the problem of pure adverse selection. In addition, peer prediction literature studies the existence of truthful Nash equilibrium in a static framework. Note that the truthful Nash equilibrium is not necessarily the maximum-benefit equilibrium in these existing peer-prediction mechanisms, i.e., the agents can gain more benefit by deviating from the truthful equilibrium.}

Note that the aforementioned literature  assumes either pure moral hazard or pure adverse selection or considers the problem of mixed moral hazard and adverse selection in a static setting. 
To the best of our knowledge,  our work is the first to study the use of reputation for mitigating both types of information asymmetry simultaneously, in particular, moral hazard followed by adverse selection. 





The main contribution of our work is  to adopt a repeated game approach to the problem of simultaneously incentivizing high effort and truthful reports in a crowdsensing setup, and more generally, in problems which exhibit moral hazard followed by adverse selection. We propose a reputation-based  payment scheme coupled with stochastic verification for compensating a sensor who realizes outcomes desired by an operator. We show that under this scheme, the sensor will exert higher effort over time, and will truthfully disclose his accuracy with a higher frequency. Furthermore, the operator needs to resort to verification with a lower frequency. Nevertheless, the operator has to provide higher payments, as a result of which her overall payoff may decrease. We discuss the intuition, as well as some practical implications of these observations. 

The remainder of the paper is organized as follows. In Section \ref{sec:model}, we present the model and some preliminaries. We analyze the proposed reputation-based payment scheme in Section \ref{sec:two-stage}, and conclude in Section \ref{sec:conclusion} with some avenues for future work.

\section{Model and preliminaries} \label{sec:model}
We study the repeated interactions of a principal (here, the system operator who is interested in estimating a quantity of interest) who contracts with an agent (here, the sensor that generates the measurements)\footnote{We will henceforth use she/her for the operator, and he/his for the sensor.}

\begin{remark}
We would like to point out that the restriction to a single sensor is without loss of generality. This is due to the fact that we assume no budget constraint for the operator and do not consider cross verification among sensors. For notational ease, we will concentrate on the case of a single sensor being present.
\end{remark}
 For our crowdsensing setup, the operator must contract with the sensor since she does not have alternate sources of measurements available. The accuracy of the measurements taken by the sensor increases with the effort expended. Thus, the operator is interested in incentivizing high effort by the sensor, so as to attain sufficiently accurate estimates. However, both the true effort exerted by the sensor, as well as the outcome he obtains in terms of the measurement or its accurracy, are unobservable by the operator. In other words, the operator faces moral hazard (in that she does not know the level of effort expended by the sensor) followed by adverse selection (in that she does not know the outcome of the effort). She therefore relies on the report by the sensor  on the accuracy of the estimation. 

Formally, at every stage $k$ ($1\leq k\leq N$), the sensor performs the following actions: 
\begin{enumerate}[(i)]
\item He exerts an effort $x_k \in [0, \bar{x}]$ to perform the task of generating a measurement for which he incurs a cost of $h(x_k)$.  The effort $x_k$ leads to an outcome accuracy level  $\alpha({x}_k)$. \item He informs the operator of the outcome accuracy level. The sensor may misreport the accuracy level to correspond to some other level of effort $\hat{x}_{k}$. 
\end{enumerate}

\begin{asu}
The accuracy $\alpha({x}_k)$ is a deterministic function of ${x}_k$, and the function is known to both the sensor and the operator. 
In other words, without loss of generality, we may assume that the sensor  reports simply his effort level to the operator. 
\end{asu}

If the report of the sensor $\hat{x}_{k}$ is equal to the actual effort $x_{k}$ expended by her, we say that the sensor has been truthful ($T$) at stage $k$;  otherwise, we say that the sensor has falsified the output and is non-truthful ($NT$). We assume at stage $k$ the sensor chooses $T$ with probability $q_k$.

The operator derives a benefit $S(x_k)$ from the task performed by the sensor, which depends on the (true) effort $x_k$ by the sensor. We assume that this function is increasing and concave. We normalize $S(0)=0$, i.e., the operator does not derive any benefit when $x_k=0$. 

The sensor does not necessarily attach value to the outcome of the task assigned to him. As a result, he should be properly incentivized to exert the effort level desired by the operator. We assume that the operator offers a payment of $P_k$ to the sensor. The problem considered in this paper is the design of this payment by the operator so that a rational self-interested sensor will take actions that lead to maximization of the operator's utility.

We now specify the utilities of the operator and the sensor. To this end, we need to discuss two factors that we will use in our payment strategy: (i) verification of the sensor's report by the operator, and (ii) the history of the past reports/efforts of the sensor, which are captured by a \emph{reputation} score assigned to the sensor.  More specifically, at each stage $k$, we assume that the operator takes an action $v_k\in \{V, NV\}$, denoting verifying and not verifying, respectively, of the sensor's reported effort $\hat{x}_k$.  We assume at stage $k$ the operator chooses $V$ with probability $p_k$.  Further, we assume that the verification is perfect and  so that the operator accurately detects any falsification by the sensor, and the operator incurs a cost of $C>0$ for this verification. 
Note that the verification is conducted on the reported effort by the sensor (e.g., by verifying the number of measurements generated by the sensor), but not on the value of the measurements.
 
  Let $z_k$ denote the level of effort known to the operator at the end of stage $k$. Therefore,
  \begin{equation}
  z_k=
  \begin{cases}
  x_k & \text{if $v_k=V$}\\
  \hat{x}_k & \text{if $v_k=NV$}.
  \end{cases}
  \label{zk}
  \end{equation}

In addition, the operator assigns a reputation $R_k$ to the sensor at each stage $k$. The reputation $R_{k}$ is updated based on the history of the sensor's past reputation scores, as well as the assumed effort $z_{k}$ at stage $k$, i.e.,
\begin{equation}
R_{k}=f(k, R_1,\cdots, R_{k-1} , z_{k})
,~~ k=1,\cdots, N,
\label{eq:rep-func}
\end{equation}
where $R_{0}=0$ and the function $f(\cdot)$ is the \emph{reputation function} selected by the operator.  
The operator then uses this reputation score to offer a compensation of $P_k:=P(R_k)$ to the sensor, where $P(\cdot)$ is an increasing function of $R_{k}$ and is the payment function that needs to be designed.

Given the above setup and compensation scheme, the utility of the sensor at stage $k$ is given by 
\begin{equation*}
U^S_{k}=P(R_{k})-h(x_{k}).
\end{equation*}
For simplicity, in this paper, we assume a linear cost of effort $h(x_k)=bx_k$, with a unit cost $b>0$. $b$ is assumed to be known to the operator a priori. Therefore, the utility of the sensor at stage $k$ reduces to
\begin{equation}
U^S_{k}= P(R_{k})-bx_k.
\end{equation}

The utility of the operator is given by 
\begin{equation}
U^P_k=S(x_k)-P(R_k) - C \mathcal{I}\{v_k=V\},
\end{equation}
where   
\[\mathcal{I}\{v_k=V\}=
\begin{cases}
1& \text{if} \:\:v_k=V\\
0& \text{otherwise}
\end{cases}.
\]
Both the sensor and the operator discount future payoffs with a factor $\delta$, so that their payoffs over the entire time horizon is given by
\[U^S = \sum_{k=1}^N \delta^k U^S_k, ~~ U^P = \sum_{k=1}^N \delta^k U^P_k~.\]

The operator has to optimize her choice of actions $\{v_1, \cdots, v_N\}$, the reputation scores $\{R_1, \cdots, R_N\}$ and the payments $\{P_1, \cdots, P_N\}$ to maximize her expected utility $\E[U^P]$ with satisfying the following constraints:
\begin{enumerate}[(i)]
\item Incentive compatibility (IC) constraint: A contract is incentive compatible if the sensor chooses to take the action preferred by the operator. {Note that the sensor is a rational decision maker, and therefore  chooses his effort level to maximize his expected  profit $\E[ U^S]$.} Formally, if $\boldsymbol{x}^*=\{x_1^*, \ldots, x_N^*\}$ denotes the sequence of  efforts desired by the operator, then the compensation scheme should be such that   $\E[U^S(\boldsymbol{x}^*)]\geq \E[U^S(\boldsymbol{x})], ~ \forall \boldsymbol{x}$.

\item  Individual rationality (IR) constraint (Participation Constraint): Both the operator and the sensor should prefer participation in the proposed scheme to opting out. Formally :
\[
\E[U^P]\geq 0,~~ \E[ U^S]\geq 0,
\]
where $\E[\cdot]$ denotes expectation. In particular, in $\E[ U^S]$, the expectation is with respect to the verification by the operator, while in  $\E[U^P]$, the expectation is with respect to the truthfulness of the sensor.  Note that both verification by the operator and truthfulness of the sensor are stochastic.
\end{enumerate}

Therefore,  the optimization the  problem for the operator is given by
\begin{equation*}
\textrm{$\mathcal{P}_1$:}
\begin{cases}
& \underset{\{v_1, \cdots,v_N\},\{R_1, \cdots,~R_N\}, \{P_1, \cdots,~P_N\} }\max \E[U^p]\\
s.t. &\textrm{IC and IR constraints}.
\end{cases}
\end{equation*}


\begin{remark}
Note that the use of verification is indispensable: if the sensor is not verified, he will always exert effort $x_k=0$, and falsify his effort as $\hat{x}_k=\bar{x}$. The goal of introducing a reputation-based payment scheme is, therefore, to reduce (but not eliminate) the verification frequency.
\end{remark}
 \begin{remark}
 While the problem we pose and the solution we introduce can be considered for a general $N$, in this paper we focus on $N=2$. This case is sufficient to illustrate  the intuition behind the general solution, and is notationally more concise.  Considering the case of $N>2$ and more interestingly the case of infinite number of stages remain for future work.
 \end{remark}

\section{ Main Results}\label{sec:two-stage}
We now proceed to consider the two-stage (i.e., $N=2$) game between the sensor and the operator and solve the problem $\mathcal{P}_1$ in this context. Specifically, we will present a linearly weighted reputation-based payment scheme. We analyze the effects of modifying the weight in this reputation function on the effort expended by the sensor, the optimal verification frequency, and the resulting utility of the operator.  

\subsection{The weighted reputation function}

As discussed in Section \ref{sec:model}, the operator assigns a reputation $R_k(\cdot)$ to the sensor at each stage $k=1,2$, and uses it to assess the payment $P(R_k)$. 
\begin{asu}
\label{assum2}
For simplicity and without loss of generality, we choose the payment function as $P(R_k)=R_k$, where $R_k$ is defined  in \eqref{eq:rep-func}. 
\end{asu}

We next propose a \emph{weighted reputation function}, in order to assess we define the history-dependent payments for the sensor. Formally, 
\begin{equation}
\begin{cases}
\text{Payment at the first stage} = R(z_1),\\
\text{Payment at the second stage} ={(1-\omega) R(z_1)+\omega R(z_2)},
\end{cases}
\label{eq:lin-w-rep}
\end{equation}
where $R(\cdot)$ is an increasing and convex function, $z_1$ and $z_2$ are defined based on \eqref{zk}, and $0\leq \omega \leq 1$ is the \emph{reputation weight}. Note that by adjusting the value of $\omega$ the operator decides on the importance of the history of the behavior of the sensor in assessing  the current payment. For instance, when $\omega=1$, the compensation is based solely on the (reported or verified) effort expended at the current stage. We refer to this special case of compensation scheme as \emph{instant payments}. 
\begin{table*}[ht]
\centering
\vskip5pt\vskip0pt
\begin{tabular}{lll}
                     &             V        &          NV           \\ \cline{2-3} 
\multicolumn{1}{l|}{T} & \multicolumn{1}{l|}{$R_k(x_k)+\gamma-bx_k$, $S(x_k)-R_k(x_k)-\gamma-C$} & \multicolumn{1}{l|}{$R_k(x_k) - bx_k$, $S(x_k)-R_k(x_k)$} \\ \cline{2-3} 
\multicolumn{1}{l|}{NT} & \multicolumn{1}{l|}{$R_k(0)$, $-R_k(0)-C$} & \multicolumn{1}{l|}{$R_k(\bar{x})$, $-R_k(\bar{x})$} \\ \cline{2-3} 
\end{tabular}
\caption{Payoffs of the Sensor (row player) and the Operator (column player) at each stage $k$}
\label{t:stage-u}
\end{table*}

We choose the following reputation scheme.
If the sensor is not verified, $R(\cdot)$ is evaluated based on the sensor's reported output $\hat{x}_k$ in that stage. If the sensor is verified and found to be truthful, he is assigned a reputation based on his verified output $x_k$, and is further added a reputation boost of $\gamma\geq 0$. If the sensor is verified and found to be non-truthful, he gets assigned the minimum reputation denoted by $l$. We set $l=0$.\footnote{A choice of $l\neq 0$ may be of interest to ensure individual rationality  when verification is not perfect.} 
We further make the following assumption.

\begin{asu}
\label{assum3}
The operator chooses the reputation function from the linear family $R(z)=\alpha z+\beta$. 
\end{asu}
Given this Assumption \ref{assum3} and the fact that $R(.)$ should be an increasing function with the minimum value $R(0)=0$ and maximum $R(\bar{x})=h$, we conclude that the reputation function is of the form $R(z)=h\frac{z}{\bar{x}}$, where $R(\bar{x})=h\geq 0$. Therefore, the parameters  $h$ and $\gamma$ are design parameters for the operator.
\subsection{The payoff matrix}
To find the payoffs of the sensor and the operator, note that we have assumed no falsification cost for the sensor. We have also assumed that the reduction in reputation due to any detected falsification is independent of the amount of falsification. As a result, if the sensor wishes to behave strategically at stage $k$, he does not exert any effort and  realizes  $x_k = 0$, but reports the maximum effort $\hat{x}_k=\bar{x}$, to gain the maximum reputation/payment if not verified. The payoff matrix of the stage game is thus specified in Table \ref{t:stage-u}.

For the game in Table \ref{t:stage-u}, the operator designs the parameters of the compensation scheme (i.e., $h$, $\gamma$, and $\omega$), to satisfy IC and IR constraints, and maximize her profit. Therefore, the optimization  problem $\mathcal{P}_1$ is refined to :
 \begin{equation*}
\textrm{$\mathcal{P}_2$:}
\begin{cases}
&\underset{\omega,h, \gamma\geq 0}\max \E[U^p]\\
s.t. &\textrm{the (IC) and (IR) constraints},\\
& \textrm{Assumption \ref{assum2}, \ref{assum3} and equation \eqref{eq:lin-w-rep}}.
\end{cases}
\end{equation*}

We now proceed to the analysis of the two-stage game for which each stage is specified  in Table \ref{t:stage-u}.

\subsection{Nash equilibria}

We start with the pure strategy Nash equilibria (NE) of the two-stage game with stage payoffs given in Table \ref{t:stage-u}.  

\begin{pro}
\label{rep-NE}
The only pure strategy Nash equilibrium of the game in Table \ref{t:stage-u} is $(NT,NV)$. In this equilibrium, the operator offers no payment to the sensor, and the sensor exerts no effort. This is equivalent to the outside option for the operator.
\end{pro}
\begin{proof}
{We start with the potential pure Nash equilibrium $(T, V)$, and analyze the payoff of the operator. By playing $V$ over $NV$ at the first stage, the operator decreases her utility by ${\gamma}$ at the subsequent stage. Therefore, $(T, NV)$ dominates $(T, V)$. Similarly, by analyzing the payoff of the sensor, we can see that $(NT, NV)$ dominates $(T, NV)$. We have therefore discarded  $(T, V)$ and  $(T, NV)$ as pure Nash equilibria of the stage games. Therefore, if a pure strategy Nash equilibrium exists, the sensor will be playing $NT$. However, given that the sensor is always playing $NT$, the operator's optimal choice is to set $h=0$, and play $NV$. Therefore, the only possible pure strategy Nash equilibrium of the game is $(NT, NV)$, in which the operator offers no payment to the sensor. Note also that this is in fact the operator's outside option.} 
\end{proof}
\vspace{0.4cm}

We, therefore, consider the mixed strategy equilibria of the game.  
The following proposition characterizes the mixed strategy NE of the game in Table \ref{t:stage-u}. 

\begin{pro}
\label{rep-mixed}
Under the weighted reputation scheme in \eqref{eq:lin-w-rep}, the mixed strategy equilibria of the game in Table \ref{t:stage-u} are as follows. The operator verifies the sensor with probabilities
\[p_2 = \frac{h-R(x_2)+\frac{1}{\omega}bx_2}{h+\gamma}~,\quad p_1 = \frac{h- R(x_1) + \frac{1}{1+(1-\omega){\delta}}bx_1}{h+\gamma}~.\] and the sensor reveals the truth with probabilities
\[q_2=\frac{h-\frac{1}{\omega}C}{h+\gamma},\quad q_1=\frac{h- \frac{1}{1+(1-\omega){\delta}}C}{h+\gamma}~.\]
For these mixed strategies to exist, the operator should select $h,\omega$ such that $\omega h>C$. 
\end{pro}
\vspace{0.4cm}
\begin{proof}
See Appendix. 
\end{proof}

\subsection{Optimal choice of payment parameters} 
We now determine the optimal choice of the parameters of the payment scheme for the operator. Recall that the operator wishes  to choose $h$, $\gamma$, and $\omega$ to maximize $\E[U^P]$, subject to the (IC) and (IR) constraints of the sensor. We start by optimizing the choice of $h$ and $\gamma$ given a fixed reputation weight $\omega$.

\begin{thm}
\label{rep-parameters} 
Consider the linearly weighted reputation function in \eqref{eq:lin-w-rep}. 
Assume that $b\bar{x}>\sqrt{CS(\bar{x})}$. Then, for a given $\omega$,
\begin{enumerate}[(i)]
\item  the optimal reputation parameters are $\gamma=0$, and $h=\frac{b\bar{x}}{\omega}$.
\item  The operator will incentivize effort level $\bar{x}$ at the first stage, and the effort $x^*$ at the second stage, where $\frac{\partial S}{\partial x}(x^*)=b$ for $\omega\neq 1$.
\item The optimal actions are as follows. The operator verifies the sensor with probabilities $p_2=1$ and $p_1=\frac{\omega}{1+(1-\omega){\delta}}$. The sensor is truthful with probabilities $q_2=1-\frac{C}{b\bar{x}}$ and $q_1=1-\frac{\omega}{1+(1-\omega){\delta}}\frac{C}{b\bar{x}}$.
\end{enumerate}
\end{thm}
\begin{proof}
See Appendix.
\end{proof}
\vspace{0.3cm}

\subsubsection{Role of inter-temporal incentives}
Intuitively, the reputation weight $\omega$ determines the importance  of  inter-temporal  incentives  (i.e.,  conditioning  future payments on the history of past efforts). In particular, $\omega=1$  yields  an  instant  payment  scheme,  in  which no  inter-temporal  incentives  are  present.  For this case,  the actions of the operator and the sensor are as follows. 

\begin{cor}
\label{no-rep-probs}
If $\omega=1$, the sensor realizes $(x^*, x^*)$,  the operator  verifies the sensor with probability $p_2=p_1=1$, and the sensor is only truthful with probability $q_1=q_2=1-\frac{C}{b\bar{x}}<1$.
\end{cor}
\begin{proof}
Note that if $\omega=1$, we have $h=\frac{b\bar{x}}{\omega}=\frac{b\bar{x}}{\delta^\omega}=b\bar{x}$ at each stage.  This leads to the sensor chooses $x^*$ at both first and second stage, i.e., $x_1=x_2=x^*$ if $\omega=1$. Further, substituting $\omega=1$ in $(iii)$ of Theorem \ref{rep-parameters}  leads to $p_2=p_1=1$, and  $q_1=q_2=1-\frac{C}{b\bar{x}}<1$.
\end{proof}

For $\omega=1$, we can see  that although the operator  always verifies the sensor, he chooses $NT$ with a strictly positive probability. It occurs due to the fact that for $\omega=1$, the utility of the sensor under $T$ is equal to his utility under $NT$. Hence, the sensor is indifferent between choice of $T$ and $NT$. 
  
By comparing Theorem \ref{rep-parameters} and Corollary \ref{no-rep-probs}, we observe that while the verification frequency, falsification probabilities, and the effort level of the senor, at the second stage with the use of reputation remain equal to the case when no reputation used, the values at the first stage are affected by the introduction of inter-temporal incentives. In particular,  when the reputations that depend on history of the behavior of the sensor are used, the operator needs to verify the sensor with a lower probability, and the sensor is truthful with a higher probability. Furthermore, the sensor exerts higher effort in the first stage.

\subsubsection{Optimal choice of the reputation weight $\omega$}

Finally, we consider the optimal choice of $\omega$, under which the operator's expected payoff is maximized. 

\begin{thm}\label{opt-w}
Assume that $b\bar{x}>\sqrt{CS(\bar{x})}$. A choice of $\omega=1$ maximizes the operator's payoff. That is, instant payments yield higher payoffs than payments based on linearly weighted reputations.  
\end{thm}
\begin{proof}
For $b\bar{x}>\sqrt{CS(\bar{x})}$, the optimal choice of $h$ is  identified in Theorem \ref{rep-parameters}. Hence, we now analyze the optimal choice of $\omega$ when $h = \frac{b\bar{x}}{\omega}$ and $(x_1,x_2)=(\bar{x},x^*)$. We need to solve the following optimization problem:
\begin{align*}
    \max_{\omega} ~~~ & -(1+\delta)C + (1-\frac{\omega}{\delta^\omega}\frac{C}{b\bar{x}})(S(\bar{x})-b\bar{x}\frac{\delta^\omega}{\omega} ) \\
&+  \delta(1-\frac{C}{b\bar{x}})(S(x^*)-b{x^*})~,~~~\text{s.t. } 0 \leq \omega\leq 1.
\end{align*}

We take the derivative of the objective function with respect to $\omega$. Define $f(\omega):=\frac{\delta^\omega}{\omega}$. Then,
\[\frac{\partial \E[U^P]}{\partial \omega} = -b\bar{x}f'(\omega)(1-\frac{1}{f^2(\omega)}\frac{S(\bar{x})C}{(b\bar{x})^2})~.\]
 With the assumption $\sqrt{CS(\bar{x})}<b\bar{x}$, and noting that $f(\omega)\geq 1$ and $f'(\omega)<0$, we conclude that $\frac{\partial \E[U^P]}{\partial \omega}>0$. Thus, $\E[U^P]$ is an increasing function of $\omega$ and the optimal choice is to set $\omega=1$. 
\end{proof}

We observe that while using inter-temporal incentives through linearly weighted reputation functions can benefit the operator by reducing the required verification frequency, increasing the effort level of the sensor, and increasing the probability of truthfulness, it will nevertheless reduce the operator's overall payoff. This is because the operator has to now offer a higher compensation to the sensor.

\section{Conclusion} \label{sec:conclusion}

In this paper, we studied the problem of contract design  between a system operator and a strategic sensor in a repeated setting. The sensor is hired to exert costly effort to collect sufficiently accurate observations for the operator. As the effort invested and the accuracy of the resulting outcome are both  private information of the sensor, the operator needs to design a compensation scheme that mitigates moral hazard followed by adverse selection. We proposed a reputation-based payment scheme coupled with stochastic verification. We showed that by increasing the importance of past behavior in our proposed linearly weighted reputation-based payments, the sensor exerts higher effort, and has a higher probability of being truthful. The operator, on the other hand, can invoke verification less frequently, but offers higher payments to the sensor, which leads to a lower payoff. 

We have so far considered inter-temporal incentives that are based on a linearly weighted reputation function. Considering other functional forms for evaluating a sensor's reputation, and its impact on the operator and sensor's strategies and payoffs, is an important direction of future work. In addition, we have considered the design of individual contracts for each  sensor, due to our assumptions of independent measurements and no budget constraint. As an interesting direction of future work, we are interested in analyzing the contract design problem for multiple sensors, given limited budget of the operator, as well as when the outcomes of the estimate at the sensor are coupled. Such coupling may enable the operator to cross-verify the outcomes of the sensors. 

\section*{Appendix}

\textit{Proof of Proposition \ref{rep-mixed}} : We use backward induction to find the operator and sensor's strategies, starting at time $k=2$. 
Assume the operator verifies the sensor with probability $p_2$. 
If the sensor reports truthfully, his expected  utility is given by
\begin{multline*}
p_2((1-\omega) R_1+\omega (R(x_2)+\gamma)-bx_2) + \\(1-p_2)((1-\omega) R_1+\omega R(x_2)-bx_2) =\\ (1-\omega) R_1+\omega R(x_2) + p_2\omega{\gamma} - bx_2~.\end{multline*}
If the sensor falsifies his report, his expected utility is given by
\begin{multline*}p_2 ((1-\omega)R_1) + (1-p_2)((1-\omega)R_1+\omega h) \\= (1-\omega){R_1} + \omega(1-p_2)h~.\end{multline*}

To make the sensor indifferent between $T$ and $NT$, the verification probability should be
\[p_2 = \frac{h-R(x_2)+\frac{1}{\omega}bx_2}{h+\gamma}~.\]

Now, assume the sensor is mixing between $T$ and $NT$ with probability $q_2$. To make the operator indifferent between $V$ and $NV$, we need
\begin{multline*}q_2(S(x_2)-((1-\omega) R_1+\omega (R(x_2)+\gamma))-C)+\\(1-q_2)(-(1-\omega)R_1-C)=q_2(S(x_2)-\\((1-\omega) R_1+\omega R(x_2)))+(1-q_2)(-((1-\omega)R_1+\omega h))
\end{multline*}
\[-q_2\omega{\gamma} - C = -(1-q_2)\omega{h} \Rightarrow q_2=\frac{h-\frac{1}{\omega}C}{h+\gamma}~.\]

Note that for the above mixed strategy to exist, the operator should choose $h$ and $\omega$ such that $\omega h>C$. Otherwise, the sensor will always play $NT$, leading to the operator playing $NV$, i.e., the outside option.

Given the above mixed strategies, the expected utility of the sensor with output $x_2$ at the second stage is given by
\begin{multline*}
\E[U^S_2(R_1, x_2)]=q_2((1-\omega){R_1}+\omega{R(x_2)} + p_2\omega{\gamma} - bx_2) \\+ (1-q_2)( (1-\omega){R_1} + \omega(1-p_2)h)\notag\\
= (1-\omega){R_1} + \frac{\omega\gamma h}{h+\gamma} + \frac{h}{h+\gamma}(\omega{R(x_2)}-bx_2)~.
\end{multline*}

Finally, the expected payoff of the operator in the second stage is given by
\begin{multline*}\E[U^P_2(R_1,x_2)] = -(1-\omega){R_1} - \omega h \frac{\gamma+\frac{1}{\omega}C}{h+\gamma}+ \\ \frac{h-\frac{1}{\omega}C}{h+\gamma}(S(x_2)-\omega{R(x_2)})\end{multline*}

We next consider the first stage. Let the probability of verification by the operator be given by $p_1$. If the sensor is truthful in this stage, he gets utility
\begin{multline*}
p_1(R(x_1)+\gamma-bx_1+\delta U^S_2(R(x_1)+\gamma, x_2)) \\+ (1-p_1) (R(x_1)-bx_1+\delta U^S_2(R(x_1), x_2)) \\
= R(x_1)-bx_1+\delta U^S_2(R(x_1), x_2) + p_1\gamma(1+(1-\omega){\delta})~.
\end{multline*}

The payoff from falsification on the other hand is given by
\begin{multline*}p_1(\delta U_2^S(0, x_2)) + (1-p_1)(h+\delta U_2^S(h, x_2)) \\= h+\delta U^S_2(h, x_2) - p_1 h (1+(1-\omega){\delta})~.
\end{multline*}

To make the sensor indifferent between the two actions, $p_1$ should be
\[p_1 = \frac{h- R(x_1) + \frac{1}{1+(1-\omega){\delta}}bx_1}{h+\gamma}~.\]

Next, let $q_1$ denote the probability that the sensor is truthful in stage 1. To make the operator indifferent between verification on not verifying, $q_1$ should be given by
\begin{multline*}
    q_1(S(x_1)-(R_1+\gamma)+\delta U^P_2(R(x_1)+\gamma, x_2)-C)+\\(1-q_1)(\delta U^P_2(0, x_2)-C)=\\q_1(S(x_1)-R_1+\delta U^P_2(R(x_1), x_2))+(1-q_1)(-h+\delta U^P_2(h, x_2))
\end{multline*}

This leads to
\[q_1=\frac{h - \frac{1}{1+(1-\omega){\delta}}C}{h+\gamma}~.\]

We need to verify that the derived $p_k$ and $q_k$ are valid probabilities. First, note that for $q_2$ to be valid, we require that $\omega h>C$.\footnote{If $\omega h<C$, the sensor will always play $NT$, in which case the operator should play $NV$, leading to the operator's outside option.} Also, as $1+(1-\omega)\delta\geq\omega$, the same assumption ensures that $q_1\geq 0$ as well. For the operator's actions, it is easy to see that $0\leq p_k\leq 1$ holds. 

Finally, note that the above analysis is valid when $x_k\neq 0$. If $x_k=0$ at either stage, the optimal strategy for the operator in that stage is to play $NV$. 
\vspace{0.4cm}

\textit{Proof of Theorem \ref{rep-parameters}} : We now proceed to finding the optimal choice of $h$ and $\gamma$ for the payment offered by the operator, under a fixed choice of reputation weight $\omega$.  

We first find the expected payoff of the sensor over the two stages of the game. The total utility of the sensor is given by
\begin{multline*}
\E[U^S(x_1,x_2)] = R(x_1)-bx_1+\\\delta \left((1-\omega){R(x_1)} + \frac{\omega\gamma h}{h+\gamma} + \frac{h}{h+\gamma}(\omega{R(x_2)}-bx_2)\right)+ \\
\frac{h- R(x_1) + \frac{1}{1+(1-\omega)\delta}bx_1}{h+\gamma}\gamma(1+(1-\omega){\delta})= (1+\delta) \frac{\gamma h}{h+\gamma} \\
+ \frac{h}{h+\gamma}((1+(1-\omega){\delta})R(x_1)+\omega{\delta}R(x_2)-bx_1 -  {\delta}bx_2)~.
\end{multline*}

We also find the expected payoff of the operator. 
\begin{small}
\begin{multline*}
\E[U^P(x_1,x_2)] = \\\frac{(h  - \frac{C}{1+(1-\omega){\delta}})\left(S(x_1)+(1+(1-\omega){\delta})(h-R(x_1))\right)}{h+\gamma}  \\-h
+ \delta \left(-(1-\omega){h} - \omega h \frac{\gamma+\frac{1}{\omega}C}{h+\gamma} + \frac{h-\frac{1}{\omega}C}{h+\gamma}(S(x_2)-\omega{R(x_2)})\right)\\
 = -(1+\delta)\frac{\gamma+C}{h+\gamma}h + \frac{h- \frac{1}{1+(1-\omega){\delta}}C}{h+\gamma}(S(x_1)-(1+(1-\omega){\delta})R(x_1)) \\+  \delta\frac{h-\frac{1}{\omega}C}{h+\gamma}(S(x_2)-\omega{R(x_2)})~.
\end{multline*}
\end{small}
First, note that with any reputation function, the derivative of the operator's utility with respect to $\gamma$ is given by
\begin{multline*}\frac{\partial \E[U^P]}{\partial \gamma} = -\frac{h-\frac{C}{\delta^\omega}}{(h+\gamma)^2}(S(x_1)+\delta^\omega(h-R(x_1)))\\ - \delta\frac{h-\frac{C}{\omega}}{(h+\gamma)^2}(S(x_2)+\omega(h-R(x_2)))<0~,\end{multline*}
where $\delta^\omega:=1+(1-\omega)\delta$. Note that $\omega\leq \delta^\omega$, with equality (only) at $\omega=1$. 
Therefore, the optimal choice is for the operator to choose $\gamma$ as small as possible (as long as the sensor's participation (IR) constraint is satisfied).

To proceed, we substitute $R(x)= h\frac{x}{\bar{x}}$. Assume the operator wants to incentivize $\hat{x}_1, \hat{x}_2$. Consider the IC constraints of the sensor. The first derivative of the sensor's utility with respect to his output level at each stage is given by
\[\frac{\partial \E[U^S]}{\partial x_1} = \frac{h}{h+\gamma}({\delta}^\omega\frac{h}{\bar{x}}-b)~, ~~~ \frac{\partial \E[U^S]}{\partial x_2} = \frac{h}{h+\gamma}{\delta}(\omega\frac{h}{\bar{x}}-b)~.\]
Also, with linear reputation functions, the utility of the sensor can be written as
\begin{multline*}\E[U^S(x_1,x_2)] = (1+\delta) \frac{\gamma h}{h+\gamma}+ \frac{h}{h+\gamma}((\frac{\delta^\omega h}{\bar{x}}-b)x_1+\\\delta(\frac{\omega h}{\bar{x}}-b)x_2)~.
\end{multline*}
Using the IC constraints, we conclude that the IR constraint of the sensor is always satisfied. As a result, we also conclude that the operator chooses $\gamma=0$. Therefore, the utility of the operator simplifies to
\begin{multline*}
    \E[U^P(x_1,x_2)] = -(1+\delta)C + (1-\frac{1}{\delta^\omega}\frac{C}{h})(S(x_1)-\\\delta^\omega h\frac{x_1}{\bar{x}}) +  \delta(1-\frac{1}{\omega}\frac{C}{h})(S(x_2)-\omega h \frac{x_2}{\bar{x}})~.
\end{multline*}

Using the IC constraints on the sensor's utility, the operator can incentivize different efforts by the sensor, depending on the choice of $h$:

\begin{itemize}
\setlength\itemsep{0.2in}
\item Case I: Set $h<  \frac{b\bar{x}}{\delta^\omega}$. Then, the sensor will realize output $0$ in both stages. Note that by Proposition \ref{rep-mixed}, the operator will choose to not verify the sensor, leading to a utility of zero. This is equivalent to the operator's outside option.  
\item Case II: Set $h = {b\bar{x}}/ {\delta^\omega}$.
In this case, the sensor will realize output $\hat{x}_2 = 0$ in the second stage, and be indifferent between all $\hat{x}_1$ in the first stage. Again by the discussion in the proof of  Proposition \ref{rep-mixed}, the operator will choose to not verify in the second stage. Therefore, her utility in this case reduces to
\[\E[U^P(x_1, 0)] = - C + (1-\frac{C}{b\bar{x}})(S(\hat{x}_1)-b{\hat{x}_1})~.\]
Note that the operator will incentivize $x_1=x^*$, for which $\frac{\partial S}{\partial x}(x^*)=b$. 
\item Case III: Set $ \frac{b\bar{x}}{{\delta}^\omega} < h < \frac{b\bar{x}}{\omega}$. Then the sensor will realize output $0$ in the second stage, and output $\bar{x}$ in the first stage. Recall also that in order to have a valid mixed strategy equilibrium, the operator has to pick $h$ such that $\omega h>C$. The operator's utility reduces to
\[\E[U^P(\bar{x}, 0)] = -C + (1-\frac{1}{\delta^\omega}\frac{C}{h})(S(\bar{x})-\delta^\omega h)~.\]
The derivative of the operator's utility with respect to $h$ is given by
\[\frac{\partial \E[U^P]}{\partial h} = \frac{CS(\bar{x})-(h\delta^\omega)^2}{h^2\delta^\omega}~~.\]
We see that if $\sqrt{CS(\bar{x})} < b\bar{x}$, the utility of the operator  is decreasing in $h$. Therefore, the optimal chocie is to set $h = \frac{b\bar{x}}{\delta^\omega}$. 
%
Note that this case becomes equivalent to Case II, but with the difference that the operator has incentivized $\bar{x}$. As $x^*$ is the optimal choice, it is easy to see that Case II dominates Case III under these parameters. 

\item Case IV: Set $h = \frac{b\bar{x}}{\omega}$. 
In this case, the sensor will realize output $\hat{x}_1 = \bar{x}$ in the first stage, and be indifferent between all $\hat{x}_2$ in the second stage. 
\begin{multline*}
    \E[U^P(\bar{x},x_2)] = -(1+\delta)C +\\ (1-\frac{\omega}{\delta^\omega}\frac{C}{b\bar{x}})(S(\bar{x})-b\bar{x}\frac{\delta^\omega}{\omega} ) +  \delta(1-\frac{C}{b\bar{x}})(S(x_2)-b{x_2})~.
\end{multline*}
The operator will incentivize $x_2=x^*$ for which $\frac{\partial S}{\partial x}(x^*)=b$.

\item Case V: Set $h > \frac{b\bar{x}}{\omega}$. Then, the sensor will realize output $\bar{x}$ in both stages. Note that with this choice, and the assumption of $b\bar{x}>C$, the constraint $\omega h>C$ is satisfied. The operator's utility reduces to:
\begin{multline*}\E[U^P(\bar{x}, \bar{x})] = -(1+\delta)C +\\ (1-\frac{1}{\delta^\omega}\frac{C}{h})(S(\bar{x})-\delta^\omega h)+  \delta(1-\frac{1}{\omega}\frac{C}{h})(S(\bar{x})-\omega h)~.
\end{multline*}
The derivative of the operator's utility with respect to $h$ is given by
\begin{multline}\frac{\partial \E[U^P]}{\partial h} = (1+\delta)\left[ \frac{CS(\bar{x}[(1-\delta)\omega+\delta]}{h^2\omega\delta^\omega}-1\right ]~,
\label{p2}
\end{multline}
Note that since $\omega<\delta^\omega$, $ b\bar{x}<h \omega$ leads to $b\bar{x} <h\delta^\omega$. Thus, for $\sqrt{CS(\bar{x})} < b\bar{x}$, we conclude that $\frac{CS(\bar{x}}{h^2\omega\delta^\omega}\leq1.$
Further given $(1-\delta)\omega+\delta \leq1$, we can see that $\E[U^P]$ is a decreasing function of $h$ and maximized if 
$h=\frac{b\bar{x}}{\omega}$. This reduces the problem to Case IV. 

\end{itemize}

Comparing the payoff of the operator in the aforementioned five cases, we conclude that, given $\omega$, the operator chooses  Case IV: $h = \frac{b\bar{x}}{\omega}$.

\vspace{0.4cm}


\bibliographystyle{IEEEtran}
%

\bibliography{ref}

\end{document}